\documentclass{article}
\usepackage{amssymb,latexsym}
\usepackage{amsmath,amsthm}
\usepackage{graphicx}
\usepackage{subfigure}
\usepackage{caption}
\usepackage{epstopdf}

\newtheorem{thm}{Theorem}[section]
\newtheorem{claim}{Claim}
\newtheorem{assm}{Assumption}
\newtheorem{obs}[thm]{Observation}
\newtheorem{cor}[thm]{Corollary}

\theoremstyle{definition}
\newtheorem{defn}[thm]{Definition}
\newtheorem{exmpl}[thm]{Example}
\theoremstyle{remark}

\numberwithin{equation}{section}

\begin{document}
\date{}

\title{Decomposition of bi-colored square arrays into balanced diagonals}

\author{
Dani Kotlar and Ran Ziv\\
\small Department of Computer Science, Tel-Hai College, Upper Galilee, Israel\\
}
\maketitle



\begin{abstract}
Given an $n\times n$ array $M$ ($n\ge 7$), where each cell is colored in one of two colors, we give a necessary and sufficient condition for the existence of a partition of $M$ into $n$ diagonals, each containing at least one cell of each color. As a consequence, it follows that if each color appears in at least $2n-1$ cells, then such a partition exists. The proof uses results on completion of partial Latin squares.
\end{abstract}
\section{Introduction}
Let $M$ be a $t\times n$ array with $t\le n$. A \emph{diagonal} in $M$ is a subset of $t$ cells of $M$ such that no two cells are in the same row or in the same column. For a natural number $k$, such that $0<k\le n$, a $k$-\emph{coloring} of $M$ is an assignment of a color from a given set of $k$ colors to each cell of $M$. Given a $k$-coloring of $M$, an $l$-\emph{transversal} ($l\le k$) is a diagonal of $M$ in which at least $l$ distinct colors are represented. We call a diagonal $D$ in a $k$-colored array $M$ \emph{balanced} if all $k$ colors appear in $D$.

A known conjecture of Stein \cite{stein75} asserts that for any $n$-coloring of an $n\times n$ array $M$, where each color appears in $n$ cells, there exists an $(n-1)$-transversal. Stein's conjecture generalizes an earlier conjecture of Ryser and Brualdi \cite{BruRys, Ryser67} which state that such a transversal exists for any $n$-coloring in which all colors in each row and each column are distinct.

A problem related to the Ryser-Brualdi-Stein Conjectures is the search for conditions allowing a decomposition of a $k$-colored $t\times n$ array into $n$ disjoint $t$-transversals. For some conjectures and asymptotic results on the subject see \cite{alon1997degree, Alon1995covering, haggkvist2008orthogonal, Hil94, pip89}.

In this paper we give a necessary and sufficient condition for a 2-colored $n\times n$ arrays to be partitioned into $n$ disjoint balanced diagonals.

\begin{defn}
We call a subset $A$ of cells in an $n\times n$ array \emph{improper} if there exists $i,j\in [n]$ such that each cell in $A$ lies either in row $i$ or in column $j$ but not in both. Otherwise, a set is called \emph{proper}.
\end{defn}

Figure~\ref{fig1} illustrates an improper set (marked with $x$'s).

\begin{figure}[h!]
\begin{center}
\includegraphics[scale=0.3]{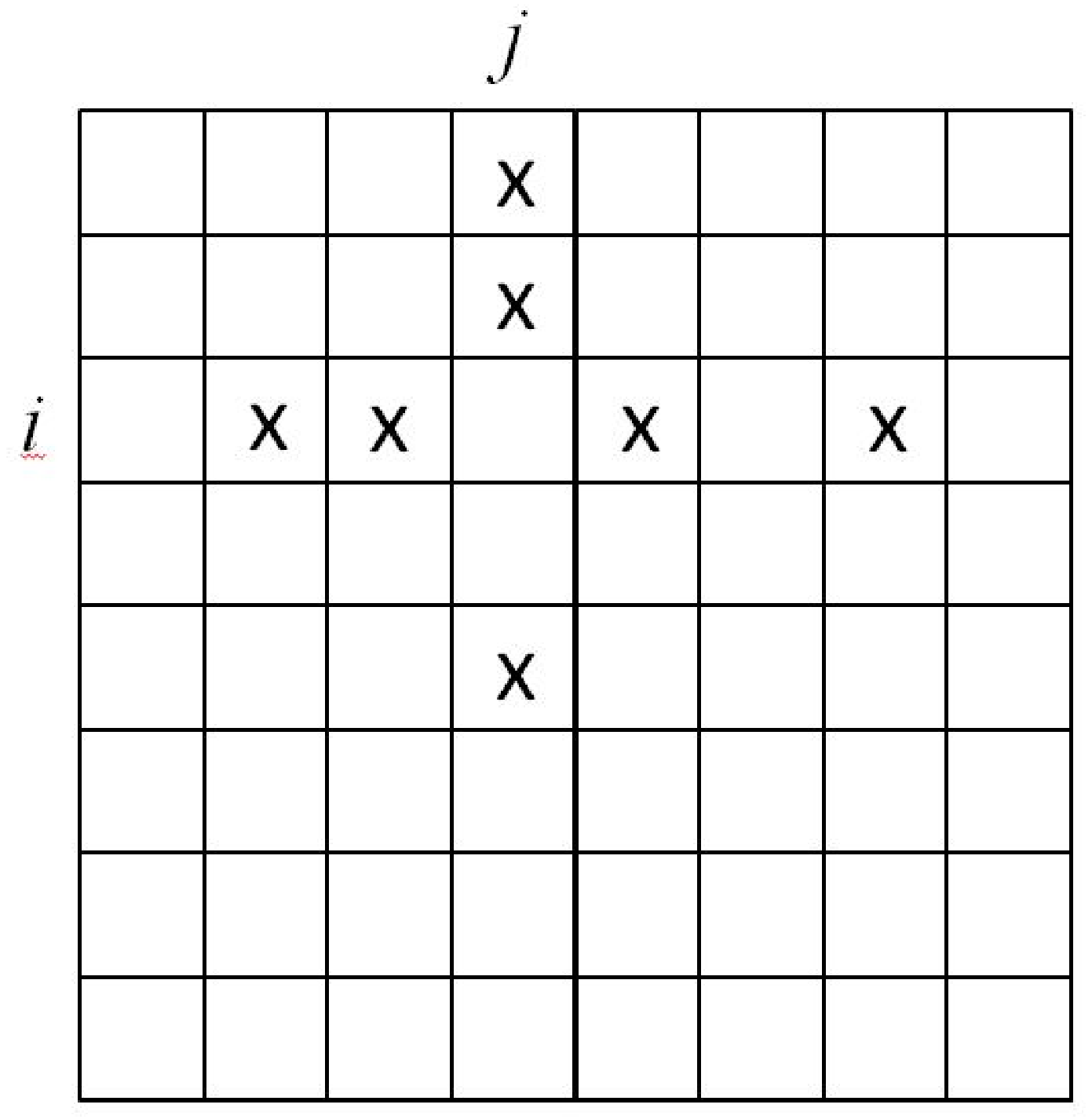}
\end{center}
\caption{}
\label{fig1}
\end{figure}

Our main result is the following theorem:

\begin{thm}\label{main:thm}
Suppose $n\ge7$ and let $M=(m_{ij})_{i,j=1}^n$ be an $n\times n$ array where each cell is colored red or blue. Then $M$ can be partitioned into $n$ balanced diagonals if and only if for each color there is a proper set of $n$ cells colored with it.
\end{thm}

The proof of Theorem~\ref{main:thm} relies upon results on completion of partial Latin squares.
\section{Completion of partial Latin squares}
A \emph{Latin square} of \emph{order} $n$ is an $n\times n$ array filled with the symbols $1,\ldots,n$ so that all symbols in each row and each column are distinct. A diagonal in a Latin square consisting of equal symbols is called a \emph{symbol diagonal}.
A \emph{partial Latin square} of \emph{order} $n$ and \emph{size} $k$ is an $n\times n$ array in which exactly $k$ cells are filled and no symbol appears more than once in a row or a column.

As a starting point for our discussion we quote the following well-known theorem, conjectured by Evans \cite{evans1960embedding} and proved by Smetaniuk \cite{smet81}. A different proof was given by Andersen and Hilton \cite{andersen1983thank}.

\begin{thm}\label{thm:smet2}
A partial Latin square of order $n$ and of size at most $n-1$ can be completed to a Latin square of order $n$.
\end{thm}

\begin{obs}
Let $M$ be an $n\times n$ array in which at least $n-1$ cells are colored blue. Then, there exists a partition of the cells of $M$ into $n$ disjoint diagonals, so that at least $n-1$ of them contain a blue cell.
\end{obs}

\begin{proof}
We assign the symbols $1,\ldots,n-1$ to the $n-1$ blue cells and obtain a partial Latin square. By Theorem~\ref{thm:smet2}, we can complete it to a Latin square in which the symbol diagonals form a partition of $M$ into diagonals, so that at least $n-1$ of them contains a blue cell.
\end{proof}

In order to explore the case where a square array contains $n$ blue cells we shall use the following theorem of Andersen and Hilton \cite{andersen1983thank}:

\begin{thm}\label{thm:hilton}
A partial Latin square of order $n$ and of size $n$ can be completed to a Latin square of order $n$, unless it can be brought by permuting rows and columns and possibly taking the transpose into one of the two forms depicted in Figure~\ref{fig2}.
\end{thm}

\begin{figure}[h!]
  \centering
  \subfigure[]{\label{fig2a}\includegraphics[scale=0.35]{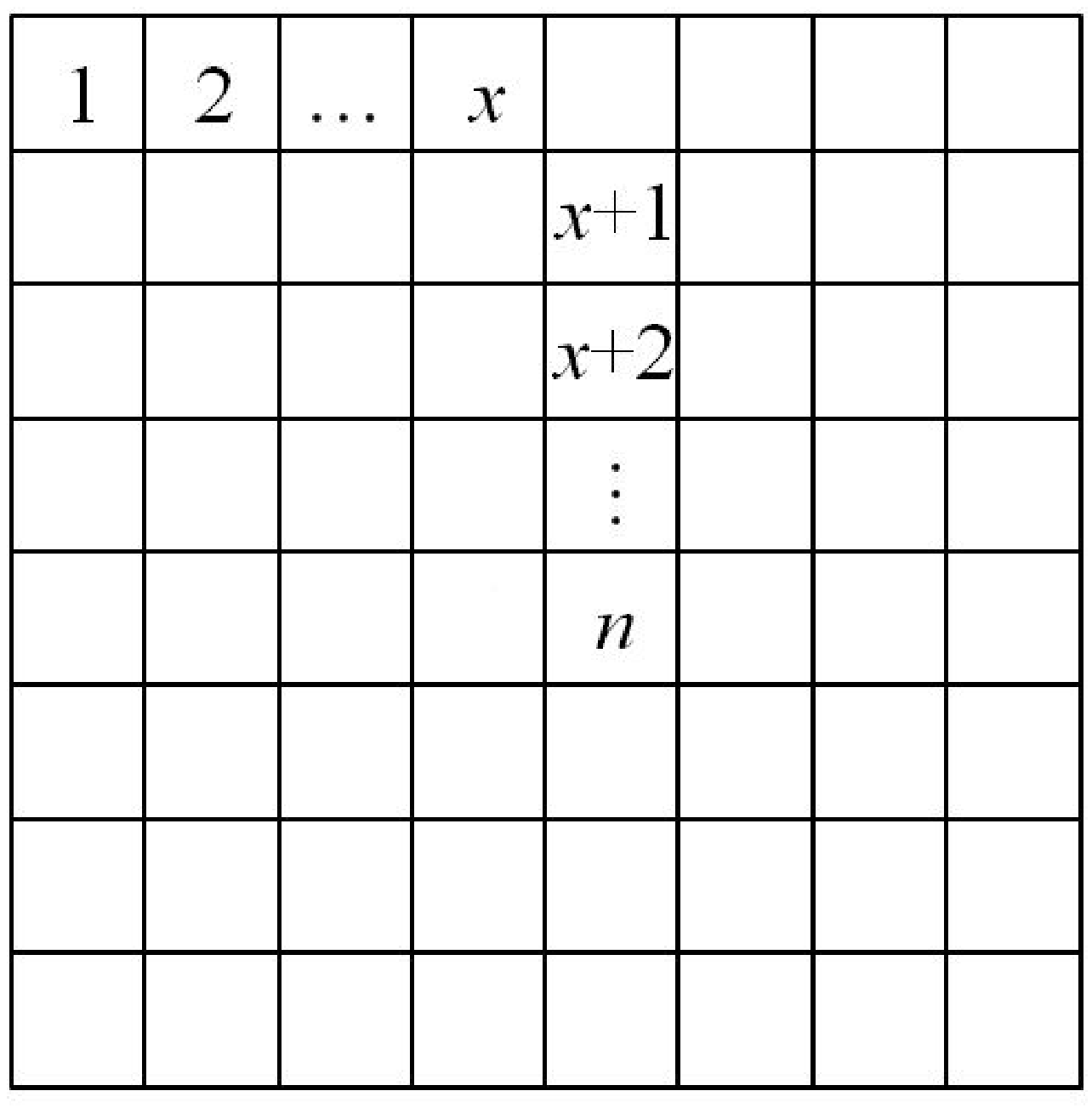}}
  \subfigure[]{\label{fig2b}\includegraphics[scale=0.35]{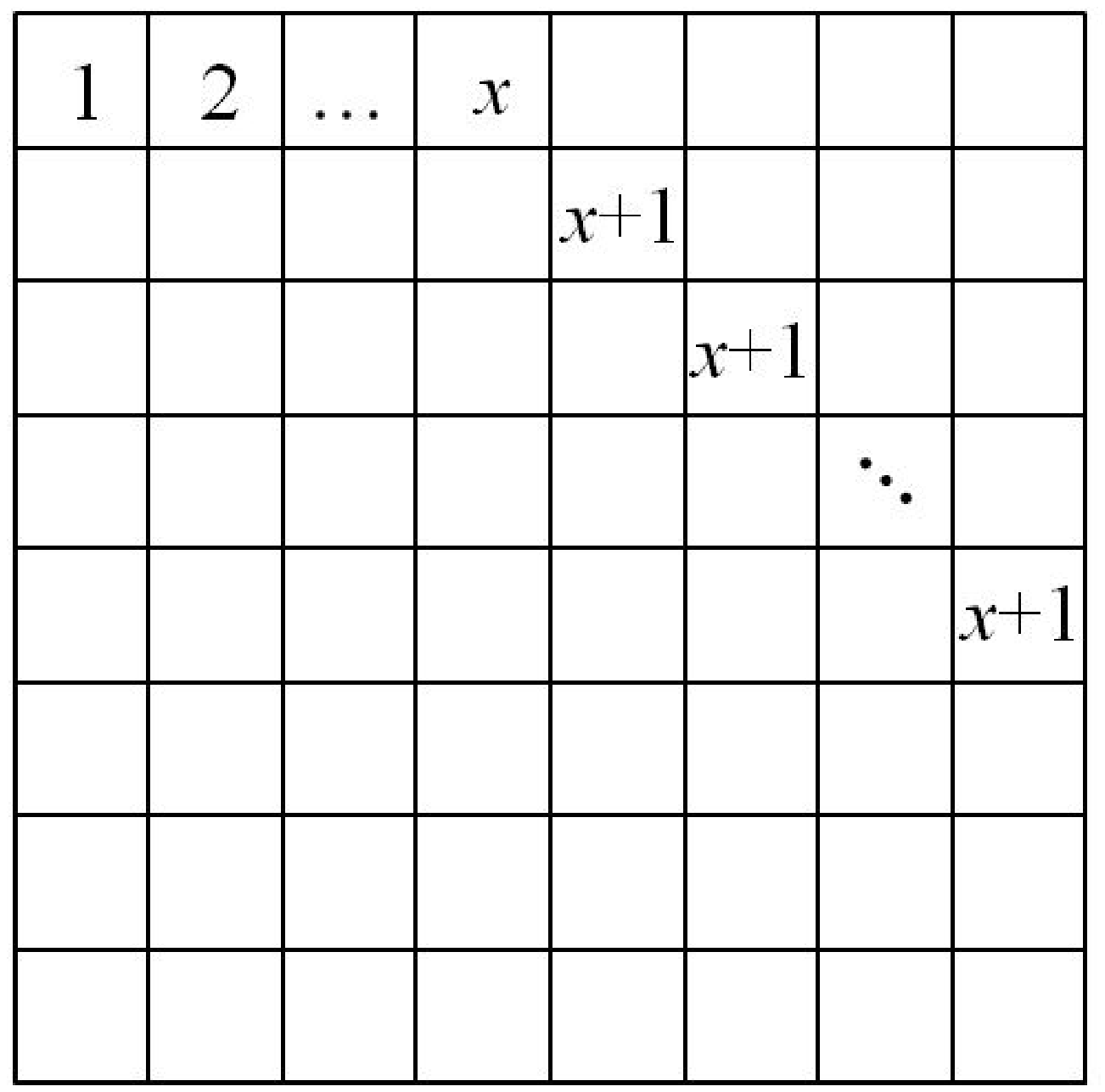}}
  \caption{}
  \label{fig2}
\end{figure}

\begin{obs}\label{cor1}
Let $M$ be a colored $n\times n$ array containing a proper subset of $n$ cells, which are all colored blue. Then, there is a partition of $M$ into $n$ diagonals, each containing a blue cell.
\end{obs}

\begin{proof}
Let $B$ be a proper set of blue cells of size $n$. We assign the symbols $1,\ldots,n$ to the cells of $B$ to obtain a partial Latin square $L$. Since $B$ is proper and properness is preserved under permutation of rows and columns and taking the transpose, it follows from Theorem~\ref{thm:hilton} that $L$ can be completed to a Latin square. The symbol diagonals of this Latin square form a partition of $M$ into diagonals, each containing a blue cell.
\end{proof}

\begin{exmpl}
The array in Figure~\ref{fig3} shows that $2n-2$ blue cells may not ensure the existence of a decomposition into diagonals, each containing a blue cell. Note that any diagonal containing the cell marked with an `x' cannot contain a blue cell.
\end{exmpl}

\begin{figure}[h!]
\begin{center}
\includegraphics[scale=0.3]{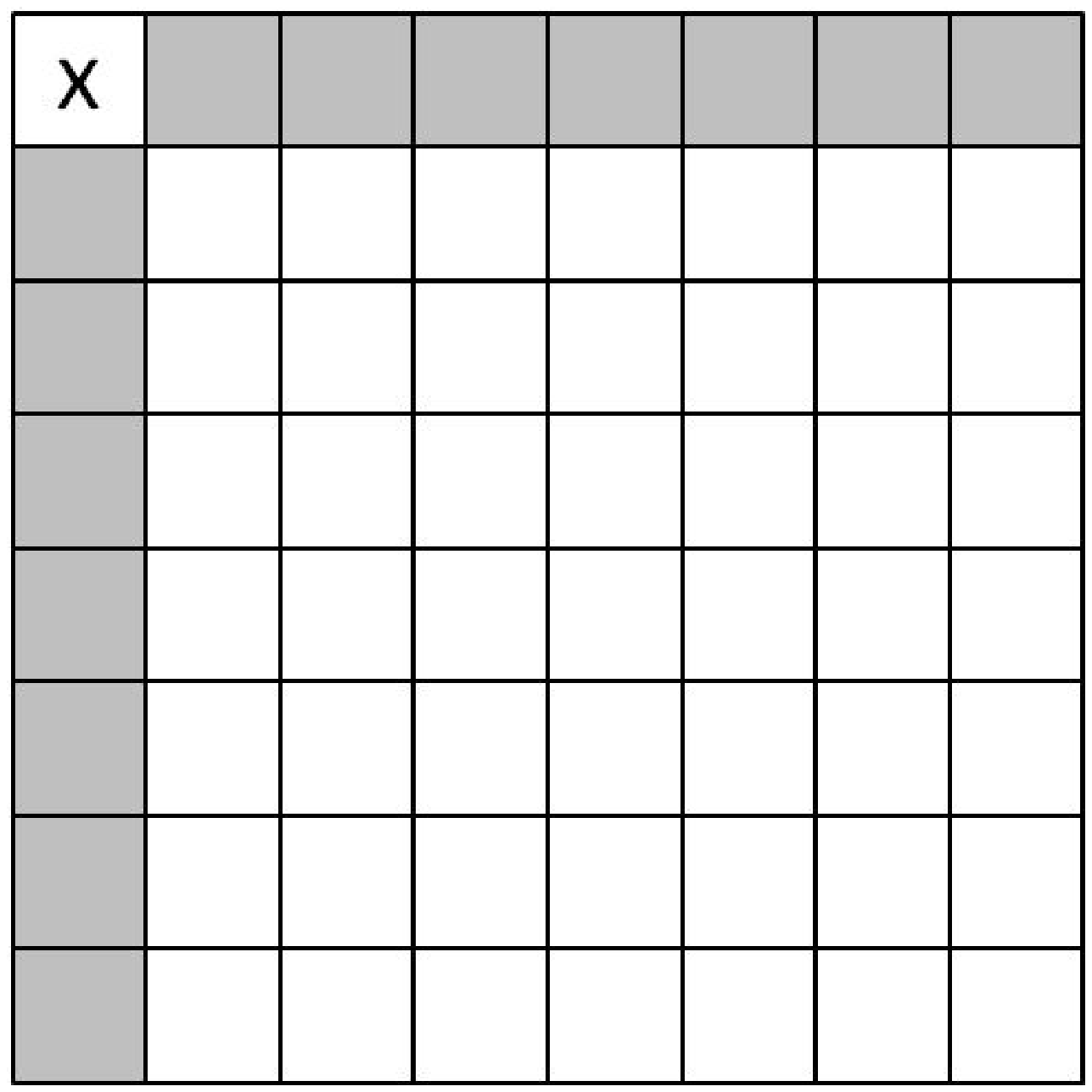}
\end{center}
\caption{}
\label{fig3}
\end{figure}

Since any set of $2n-1$ cells is proper, and thus contains a proper subset of size $n$, we have the following observation:

\begin{obs}\label{cor2}
Let $M$ be a $n\times n$ array in which at least $2n-1$ cells are colored blue. Then, there is a partition of $M$ into $n$ diagonals, each containing a blue cell.
\end{obs}


\section{Proof of the main result}

For the proof of Theorem~\ref{main:thm} we shall need the following theorem of Ryser \cite{ryser51}:

\begin{thm}\label{thm:rys}
Let $0< r,s<n$ and let $A$ be a partial Latin square of order $n$ in which cell $(i,j)$ in $A$ is filled if and only if $i\le r$ and $j\le s$. Then $A$ can be completed to a Latin square if and only if $N(i)\ge r+s-n$ for $i=1,\ldots,n$, where $N(i)$ is the number of cells in $A$ that are filled with $i$.
\end{thm}

\setcounter{claim}{0}

\begin{proof}[Proof of Theorem~\ref{main:thm}]
Let $M_b$ and $M_r$ be the subsets of $M$ consisting of blue and red cells, respectively. Without loss of generality we may assume that $|M_b|\le |M_r|$. If $|M_b|<n$, then clearly there is no decomposition of $M$ into balanced diagonals. Hence, we may assume that $|M_b|\ge n$. If $M_b$ does not contain a proper subset of size $n$, then $M_b$ is improper. Suppose $M_b$ is contained in row $i$ and column $j$, then for any partition of $M$ into diagonals, the diagonal through $m_{ij}$ will be contained in $M_r$. Thus, the condition is necessary.

We now show that the condition of the theorem is sufficient. For contradiction, we make the following assumption:

\begin{assm}\label{assm:1}
A decomposition of $M$ into balanced diagonals does not exist.
\end{assm}

The proof consists of a sequence of claims.

\begin{claim}\label{claim:1}
$M_b$ contains two diagonals $T_1$ and $T_2$ such that $|T_1\cap T_2|=1$.
\end{claim}

\begin{proof}[Proof of Claim~\ref{claim:1}]
\renewcommand{\qedsymbol}{}
By Observation~\ref{cor1}, there exists a decomposition of $M$ into diagonals, each containing a blue cell. By Assumption~\ref{assm:1}, at least one of these diagonals is contained in $M_b$. We denote this diagonal by $T_1$. If we put all the symbols $\{1,\ldots,n\}$ in $T_1$, then clearly we have a partial Latin square that can be completed to a Latin square $L'$.
By Assumption~\ref{assm:1}, at least one of the symbol diagonals of $L'$ is contained in $M_b$. Let $T_2$ be one such diagonal. We have $T_1\cup T_2\subset M_b$ and $|T_1\cap T_2|=1$, since $T_1$ contains the distinct symbols $1,\ldots,n$ and $T_2$ contains the same symbol in all its cells.
\end{proof}

\begin{claim}\label{claim:2}
Let $T_1\cap T_2=\{m_{ij}\}$. Then, there exists a cell in $M_b\setminus (T_1\cup T_2)$ which is not in row $i$ and not in column $j$.
\end{claim}

\begin{proof}[Proof of Claim~\ref{claim:2}]
\renewcommand{\qedsymbol}{}
Note that we can choose $\lceil n/2 \rceil$ columns $C_1,\ldots C_{\lceil n/2 \rceil}$ of $M$, such that $D=(\cup_{i=1}^{\lceil n/2 \rceil} C_i)\cap (T_1\cup T_2)$ has size $n$ (if $n$ is even we take any set of $n/2$ columns that does not include column $j$. If $n$ is odd we take any set of $\lceil n/2 \rceil$ columns that includes column $j$).
If we put the symbols $1,\ldots,n$ in the cells of $D$, we can complete to a Latin square $L''$, by Theorem~\ref{thm:hilton}. By Assumption~\ref{assm:1}, $L''$ must have a symbol diagonal $T_3$ which is contained in $M_b$. Note that $|T_3\cap D|=1$ since $D$ contains distinct symbols and $T_3$ has the same symbol in all its cells. We look at the set $T_3\cap (\cup_{i=1}^{\lceil n/2 \rceil}  C_i)$, which is of size $\lceil n/2 \rceil$. It contains one cell of $D$, possibly one cell from row $i$ and possibly one cell from column $j$. Thus, the set of cells $T_3\cap (\cup_{i=1}^{\lceil n/2 \rceil}  C_i)$ contains at least $\lceil n/2\rceil-3$ cells which are neither in $T_1\cup T_2$ nor in column $i$ nor in column $j$. This number is positive since we assumed $n\ge 7$.
\end{proof}

\begin{claim}\label{claim:3}
The array $M$ contains an $s\times t$ sub-rectangle $R_1$, such that $s+t= n$, $s-1\le t\le s+1$ and $|R_1\cap M_b|\ge n$.
\end{claim}
\begin{proof}[Proof of Claim~\ref{claim:3}]
\renewcommand{\qedsymbol}{}
Let $T_1\cap T_2=\{m_{ij}\}$ and let $m_{kl}\in M_b\setminus (T_1\cup T_2)$, as in Claim~\ref{claim:2}, that is $k\ne i$ and $\l\ne j$. We regard $T_1$ and $T_2$ as two perfect matchings in $K_{n,n}$ and consider the subgraph $G$ of $K_{n,n}$ consisting of the edges in $(T_1\cup T_2)\setminus \{m_{ij}\}$. Since $T_1\cap T_2=\{m_{ij}\}$, it follows that $G$ is the disjoint union of simple even cycles, each of length at least 4. We make the following two observations:

\begin{obs}\label{obs1}
For any $k\le 2n-2$, every induced subgraph of $G$ with $k$ edges, consisting of the union of cycles and possibly one path, has at most $k+1$ vertices.
\end{obs}

\begin{obs}\label{obs2}
For any two vertices $u$ and $v$ in $G$, there exists an induced subgraph $H$ of $G$ containing $u$ and $v$, consisting of the union of cycles and possibly one path, such that $|E(H)|=n-1$.
\end{obs}

The proof of Observation~\ref{obs1} is left to the reader. We prove Observation~\ref{obs2}.
\begin{proof}[Proof of Observation~\ref{obs2}]
\renewcommand{\qedsymbol}{}
First assume $u$ and $v$ are in the same simple cycle $C$ of $G$. If $|C|\le n-1$ we take $C$ and add cycles and possibly one path (contained in a simple cycle) in $G$ to obtain $H$ as required. If $|C|> n-1$ we can take $H$ to be a path in $C$ containing $u$ and $v$. Since $|E(G)|=2n-2$, such a path $H$ with $n-1$ edges always exists.

Now, assume $u$ and $v$ lie in disjoint cycles $C_u$ and $C_v$ of $G$, respectively. Since $|E(G)|=2n-2$, we may assume, without loss of generality, that $C_u$ has size at most $n-1$. If the size of $C_u\cup C_v$ is greater than $n-1$ we take $C_u$ and a path containing $v$ from $C_v$ to obtain $H$ as required (in the case that $C_u$ has size exactly $n-1$ we just add the path of length 0 consisting of $v$). In case the size of $C_u\cup C_v$ is less than $n-1$ we take $C_u\cup C_v$ and add possibly more cycles and possibly one path from $G$ to obtain $H$.
\end{proof}

Let $e$ be the edge corresponding to $m_{kl}$. Since $k\ne i$ and $\l\ne j$ and $G$ consists of the edges in $(T_1\cup T_2)\setminus \{m_{ij}\}$, the endpoints of $e$ are in $V(G)$. Let $u$ and $v$ be the endpoints of $e$. By Observations~\ref{obs1} and \ref{obs2}, there is an induced subgraph $H$ of $G$ containing $u$ and $v$, such that $|E(H)|=n-1$ and $n-1\le|V(H)|\le n$. Let $s$ and $t$ be the sizes of the two sides of $H$. We have $n-1\le s+t\le n$ and, since $H$ is the union of cycles and possibly one path, we must have that $s$ and $t$ differ by at most 1. The graph $H\cup\{e\}$ corresponds to a sub-rectangle of $M$ of size $s\times t$. If $s+t=n-1$ we augment this rectangle by a row or a column to satisfy $s+t=n$ and $s-1\le t\le s+1$.
\end{proof}

\begin{claim}\label{claim:4}
$M$ contains a $p\times q$ sub-rectangle $R$, such that $p+q=n+1$, $p-2\le q\le p+2$, $|R\cap M_b|\ge n$ and $|R\cap M_r|\ge n$.
\end{claim}
\begin{proof}[Proof of Claim~\ref{claim:4}]
\renewcommand{\qedsymbol}{}

Let $R_1$ be a sub-rectangle of $M$ as in Claim~\ref{claim:3}. If $R_1$ contains also $n$ red cells, then we are done. Otherwise, since $|M_b|\le |M_r|$, the square $M$ must contain another $s\times t$ sub-rectangle $R_2$ with a majority of red cells and, since $n\ge 7$, $|R_2\cap M_r|\ge n$. If $R_2$ contains $n$ blue cells we are done, so we assume it does not. We can travel from $R_1$ to $R_2$ using an $s\times t$ sliding window (Figure~\ref{fig4}), so that in each step we either drop a row and add a row or drop a column and add a column. Clearly, at some stage, by exchanging a single row or a single column we shall move from a rectangle $R'$ containing a majority of blue cells to a rectangle $R''$ containing a majority of red cells. Let $R=R'\cup R''$.  Clearly, $R$ is a rectangle containing $n$ blue cells and $n$ red cells and its size is $p\times q$ satisfying $p+q=s+t+1=n+1$ and $p$ and $q$ differ at most by 2.

\begin{figure}[h!]
\begin{center}
\includegraphics[scale=0.4]{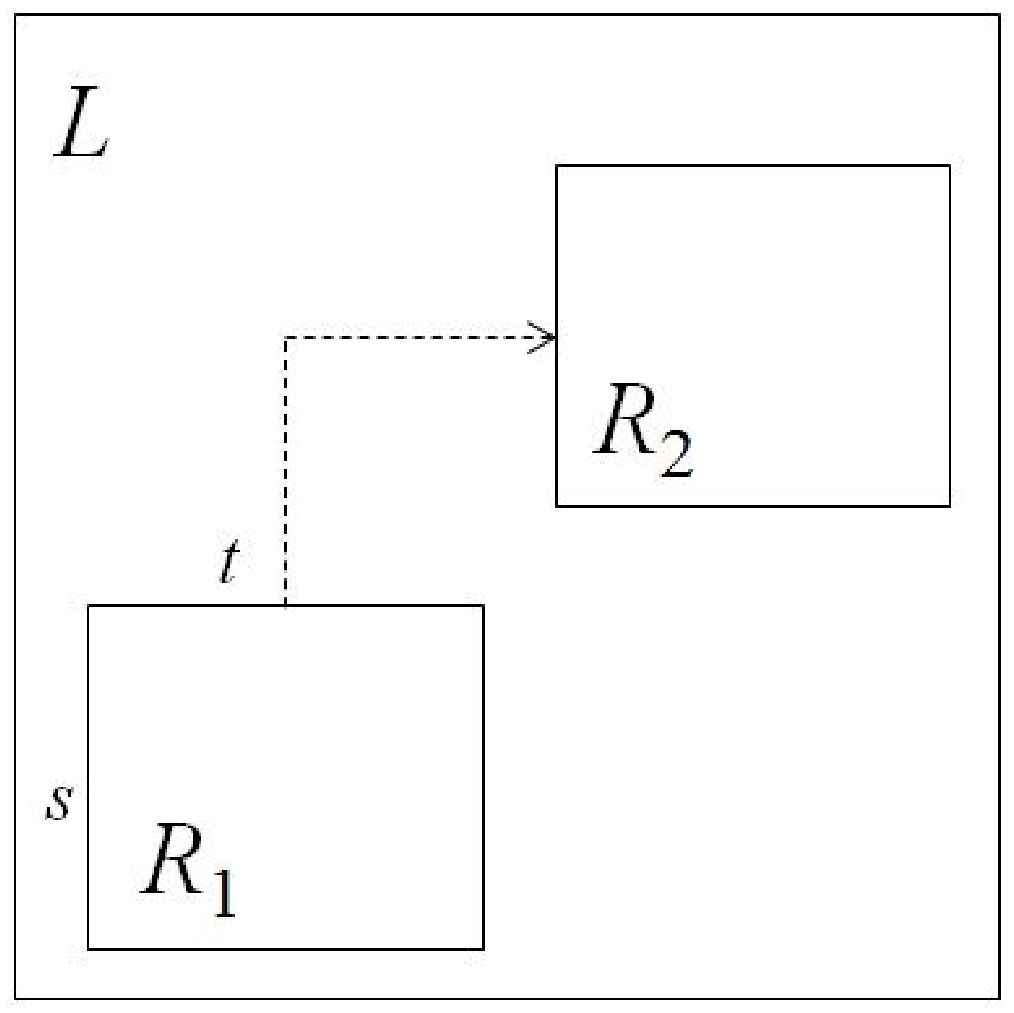}
\end{center}
\caption{}
\label{fig4}
\end{figure}

\end{proof}

\begin{claim}\label{claim:5}
Let $R$ be a sub-rectangle of $M$ as in Claim~\ref{claim:4}. Then, we can fill $n$ blue cells and $n$ red cells of $R$ with the numbers $1,\ldots,n$ so that each number appears once in a blue cell and once in a red cell, to form a partial Latin square.
\end{claim}

\begin{proof}[Proof of Claim~\ref{claim:5}]
\renewcommand{\qedsymbol}{}
Let $X$ be a set of $n$ blue cells in $R$ and let $Y$ be a set of $n$ red cells in $R$.  We form a bipartite graph $G$ whose sides are the sets $X$ and $Y$ and we draw an edge between two vertices if and only if the corresponding cells are neither in the same row nor in the same column. We apply Hall's theorem to show that there is a perfect matching in $G$. Let $S\subset X$ such that $|S|=k$.
Assume first that all the cells of $S$ are in the same row or in the same column. Since the largest side of $R$ is of size at most $\frac{n+3}{2}$, then, $|N(S)|\ge n-(\frac{n+3}{2}-k)$. If we assume, for contradiction, that $|N(S)|<k$, then we get $k> n-(\frac{n+3}{2}-k)$ which leads to $n<3$. Thus, we may assume that the cells of $S$ are not all in the same row or in the same column. Note that in this case, $|N(S)| = n$ unless the cells of $S$ form an improper set, in which case $|N(S)| = n-1$. But, the size of an improper set in $R$ is at most $(p-1)+(q-1)=n-1$. Thus, Hall's condition holds and the desired pairing exists.
\end{proof}

We can now conclude the proof of Theorem~\ref{main:thm}.
Let $R$ be a sub-rectangle of $M$ as in Claim~\ref{claim:4}. We fill $n$ blue cells with $1,\ldots n$ and $n$ red cells with $1,\ldots n$, as in Claim~\ref{claim:5}. Since for each cell in $R$ there are $(p-1)+(q-1)=n-1$ other cells in $R$ which are in the same row or in the same column and there are $n$ symbols, all the cells in $R$ can be filled to yield a partial Latin square $L$. We have $p+q-n=1$ and each symbol appears at least twice in $R$. By Theorem~\ref{thm:rys}, $L$ can be completed to a Latin square $L'$. Since each of the symbols $1,\ldots,n$ appears in a blue cell and in a red cell of $R$, all the symbol diagonals of $L'$ are balanced. This completes the proof of Theorem~\ref{main:thm}.

\end{proof}

Since any set of $2n-1$ cells contains a proper subset of size $n$ we have the following corollary:

\begin{cor}\label{cor3}
Let $M$ be a 2-colored $n\times n$ array with $n\ge7$. If each color appears in at least $2n-1$ cells, then $M$ can be partitioned into $n$ balanced diagonals.
\end{cor}

The results in this paper originated from questions on edge colorings of the complete bipartite graph $K_{n,n}$. Thus, we formulate Corollary~\ref{cor3} in these terms.

\begin{defn}
Let $f:E(K_{n,n})\rightarrow \{1,2\}$ be a coloring. A matching in $P\subset E(K_{n,n})$ is called \emph{balanced} if $f^{-1}(i)\cap P\ne\emptyset$ for $i=1,2$.
\end{defn}

\begin{thm}\label{thm2}
Let $n\ge 7$ and let $f:E(K_{n,n})\rightarrow \{1,2\}$ be a coloring. If $f^{-1}(i)\ge 2n-1$ for $i=1,2$, then there exists a partition of $E(K_{n,n})$ into $n$ disjoint balanced matchings.
\end{thm}


\end{document}